\documentclass[12pt]{amsart}
\usepackage{amssymb,amsmath,amsfonts,amsthm,mathrsfs}

\usepackage[dvipdfm]{graphicx,color}

\usepackage{hyperref}
\renewcommand\eqref[1]{(\ref{#1})} 

\numberwithin{equation}{section}  
\theoremstyle{plain}
\newtheorem{thm}{Theorem}[section]

\newtheorem{cor}[thm]{Corollary}
\newtheorem{lem}{Lemma}[section]

\theoremstyle{definition}
\newtheorem{defn}[thm]{Definition}

\theoremstyle{remark}

\def\R{{\mathbb R}}
\def\C{{\mathbb C}}

\def\L2tx{{L^2(\R_t\times\R^n_x)}}

\def\p#1{{\left({#1}\right)}}
\def\b#1{{\left\{{#1}\right\}}}

\def\n#1{{\left\|{#1}\right\|}}
\def\abs#1{{\left|{#1}\right|}}
\def\jp#1{{\left\langle{#1}\right\rangle}}

\def\supp{\operatorname{supp}}

\def\A{{\mathcal A}}

\title[Global regularity properties of Fourier integral operators]
{Global regularity properties for \\ a class of Fourier integral operators}

\author[Michael Ruzhansky and Mitsuru Sugimoto]{Michael Ruzhansky and Mitsuru Sugimoto}
\address{
  Michael Ruzhansky:
  \endgraf
  Department of Mathematics
  \endgraf
  Imperial College London
  \endgraf
  180 Queen's Gate, London SW7 2AZ, UK
  \endgraf
  {\it E-mail address} {\rm m.ruzhansky@imperial.ac.uk}
  \endgraf
  \medskip
 Mitsuru Sugimoto:
  \endgraf
  Graduate School of Mathematics
  \endgraf
  Nagoya University
  \endgraf
  Furocho, Chikusa-ku, Nagoya 464-8602, Japan
  \endgraf
  {\it E-mail address} {\rm sugimoto@math.nagoya-u.ac.jp}
  }
\thanks{The first
 author was supported in parts by the EPSRC
 grant EP/K039407/1 and by the Leverhulme Grant RPG-2014-02.
}

\subjclass{Primary 35S30; Secondary 35B65, 35L40}
\keywords{Fourier integral operators, $L^p$-estimates, hyperbolic equations}
\date{\today}
\dedicatory{Dedicated to the memory of Professor Hans Duistermaat (1942--2010)}

\begin{document}

\begin{abstract}
While the local $L^p$-boundedness of nondegeneral Fourier integral operators is known
from the work of Seeger, Sogge and Stein \cite{SSS}, not so many results are available
for the global boundedness on $L^p(\R^n)$.
In this paper we give a sufficient condition for
the global $L^p$-boundedness for a class of Fourier integral operators which
includes many natural examples. We also describe a construction that is used to
deduce global results from the local ones. An application is given to obtain global $L^p$-estimates
for solutions to Cauchy problems for hyperbolic partial differential equations.
\end{abstract}

\maketitle

\section{Introduction}\label{S1}
In this article, we discuss the global $L^p$-boundedness
of the Fourier integral operators
\[
\mathcal P u(x)
=\int_{\R^n}\int_{\R^n} e^{i\phi(x,y,\xi)}
a(x,y,\xi) u(y)\, dy  d\xi\quad(x\in \R^n).
\]
We always assume that $n\geq 1$ and $1<p<\infty$.
Here $\phi(x,y,\xi)$ is a real-valued function that is called a {\it phase function}
while $a(x,y,\xi)$ is called an {\it amplitude function}.
Following the theory of Fourier integral operators by H\"ormander \cite{Ho}, 
we originally assume that $\phi(x,y,\xi)$ is positively homogeneous
of order $1$ and smooth at $\xi\neq0$,
and that $a(x,y,\xi)$ is smooth and satisfies
a growth condition in $\xi$ with some $\kappa\in\R$:
\[
\sup_{(x,y)\in K}\abs{\partial_x^\alpha\partial_y^\beta\partial_\xi^\gamma
 a(x,y,\xi)}
\leq C_{\alpha\beta\gamma}^K\jp{\xi}^{\kappa-|\gamma|}\quad
(\forall \alpha, \beta, \gamma)\,;\quad 
\langle \xi\rangle=\p{1+|\xi|^2}^{1/2}
\]
for any compact set $K\subset\R^n\times\R^n$.
Then the operator $\mathcal P$ is just a microlocal
expression of the corresponding Lagrangian manifold, and with
the local graph condition, it is microlocally equivalent to the special form
\begin{equation}\label{EQ:P}
P u(x)=
\int_{\R^n}\int_{\R^n} e^{i(x\cdot\xi-\varphi(y,\xi))}a(x,y,\xi)u(y)\, dy d\xi
\end{equation}
by an appropriate microlocal change of variables.

\medskip
The local $L^p$ mapping properties of Fourier integral operators
have been extensively studied, and can be generally summarised as follows:
\begin{itemize}
\item $\mathcal P$ is $L^2_{comp}$-$L^2_{loc}$-bounded
when $\kappa\leq 0$ (H\"ormander \cite{Ho}, Eskin \cite{Es});
\item $\mathcal P$ is $L^p_{comp}$-$L^p_{loc}$-bounded
when $\kappa\leq -(n-1)|1/p-1/2|$, $1<p<\infty$ (Seeger, Sogge and Stein \cite{SSS});
\item $\mathcal P$ is $H^1_{comp}$-$L^1_{loc}$-bounded
when $\kappa\leq -\frac{n-1}{2}$ (Seeger, Sogge and Stein \cite{SSS}), where
here and everywhere $H^1=H^1(\R^n)$ is the Hardy space introduced by Fefferman and Stein \cite{FS};
\item $\mathcal P$ is locally weak $(1,1)$ type when $\kappa\leq -\frac{n-1}{2}$ 
(Tao \cite{Tao:weak11}).
\end{itemize}

The sharpness of order the $-(n-1)|1/p-1/2|$ was shown by Miyachi \cite{Mi}
and Peral \cite{Pe}
(see also \cite{SSS}). 
Therefore, the question addressed in this paper is when Fourier integral operators are globally
$L^p$-bounded.
Although the operator $\mathcal P$ or $P$
is just a microlocal expression of
the corresponding Lagrangian manifold due to the Maslov cohomology class
(see e.g. Duistermaat \cite{Duistermaat:FIO-book-1996}), 
we still regard it as a globally defined operator
since it is still important for the applications to the theory of
partial differential equations.
Indeed, the operator $P$ is used to:
\begin{itemize}
\item
express solutions to Cauchy problems of hyperbolic equations;
\item
transform equations to another simpler one (Egorov's theorem).
\end{itemize}
For example, the solution to the wave equation
\begin{equation*}
\left\{
\begin{array}{ll}
\partial_t^2 u(t,x) - \Delta u(t,x) = 0,\ t \in \R,\ x \in \R^n, \\
u(0,x) = g(x),\ \partial_t u(0,x) = 0,\ x \in \R^n,
\end{array}
\right.
\end{equation*}
is expressed as a linear combination of the operators of the form
\[
 P g(x)=\int_{\R^n} \int_{\R^n} e^{i(x\cdot\xi-y\cdot \xi \pm t|\xi|)} g(y) \, dyd\xi.
\]
On the other hand,
we have the relation
\[
 P\cdot\sigma(D)=\p{\sigma\circ\psi}(D)\cdot P
\]
if we take $a(x,y,\xi)=1$ and
$\varphi(y,\xi)=y\cdot\psi(\xi)$
so that we can 
transform\footnote{Microlocally, this idea was explored by Duistermaat and
H\"ormander \cite{Duistermaat-Hormander:FIOs-2} in a variety of problems, while in the global analysis it
was applied by the authors in \cite{RS2,RS4} to the study of the global smoothing estimates.}
the operator $\sigma(D)$ to
$\p{\sigma\circ\psi}(D)$
which might have been very well investigated.
Summarising these situations, the typical two types of phase functions for each analysis are
\[
\text{
(I)\hspace{5mm} $\varphi(y,\xi)=y\cdot\xi+|\psi(\xi)|$,
\hspace{10mm}
(II)\hspace{5mm}  $\varphi(y,\xi)=y\cdot\psi(\xi)$,}
\]
where $\psi(\xi)$ is a real vector-valued smooth function which is
positively homogeneous of order $1$ for large $\xi$.
(See Definition \ref{Def:hom} for the precise meaning of this terminology).

\medskip
A global $L^2$-boundedness result of the operator $P$ with
the phase function (I) was given by
Asada and Fujiwara \cite{AF} which states it for rather general operators
$\mathcal P$:
\begin{thm}[\cite{AF}]\label{Th:AF}
Let $\phi(x,y,\xi)$ and $a(x,y,\xi)$ be $C^\infty$-functions, and let
\[
D(\phi):=
\begin{pmatrix}
\partial_x\partial_y\phi&\partial_x\partial_\xi\phi
\\
\partial_\xi\partial_y\phi&\partial_\xi\partial_\xi\phi
\end{pmatrix}.
\]
Assume that
$|\det D(\phi)|\geq C>0$.
Also assume that every entry of the matrix $D(\phi)$, $a(x,y,\xi)$
and all their derivatives are bounded. 
Then $\mathcal P$ is $L^2(\R^n)$-bounded.
\end{thm}
The result of \cite{AF} was used to construct
the solution to the
Cauchy problem of Schr\"odinger equations by means of
the Feynman path integrals in Fujiwara \cite{Fu}.
For the operator $P$, the conditions of Theorem \ref{Th:AF} are reduced to
a global version of the {\it local graph condition}
\begin{equation}\label{graphcond}
\abs{\det \partial_y\partial_\xi\varphi(y,\xi)}\geq C>0,
\end{equation}
and the growth conditions
\begin{align*}
&\abs{\partial_y^\alpha\partial_\xi^\beta\varphi(y,\xi)} \le
C_{\alpha\beta}\quad
(\forall\, |\alpha+\beta|\geq2, |\beta|\geq1),
\\
&\abs{\partial_x^\alpha\partial_y^\beta\partial_\xi^\gamma a(x,y,\xi)}
\le C_{\alpha\beta\gamma}\quad
(\forall \alpha, \beta, \gamma),
\end{align*}
for all $x,y,\xi\in\R^n$. Note that the phase function (I) satisfies these conditions.
We also note that the condition 
\eqref{graphcond} is required even for the local $L^2$-boundedness
of Fourier integral operators of order zero, so it is rather natural to assume it to hold 
globally on $\R^n$ as well.

We remark that Kumano-go \cite{Ku} also showed the same conclusion as that of Theorem \ref{Th:AF}
under weaker conditions on the phase function,
namely for 
\[
\abs{\partial_y^\alpha\partial_\xi^\beta (\varphi(y,\xi)-y\cdot\xi)}
\le C_{\alpha\beta}\jp{\xi}^{1-|\beta|}
\quad (\forall \alpha, \beta),
\]
with applications to the global $L^2$ estimates for solutions to Cauchy problems of
strictly hyperbolic equations.

\medskip
Unfortunately the phase function (II) does not satisfy the growth condition
in Theorem \ref{Th:AF} because $\partial_\xi\partial_\xi\varphi$ are usually
unbounded.
Therefore, another type of conditions was introduced by the authors in \cite{RS}
to obtain the global $L^2$-boundedness for operators with phases of the type
(II). Such $L^2$-boundedness results were then used to show global smoothing estimates
for dispersive equations in a series of papers \cite{RS2},
\cite{RS3} and \cite{RS4}. Thus, the result covering the case (II) is as follows:

\begin{thm}[\cite{RS}]\label{Th:RS}
Let $\varphi(y,\xi)$ and $a(x,y,\xi)$ be $C^\infty$-functions.
Assume \eqref{graphcond}.
Also assume that
\begin{align*}
&\abs{\partial_y^\alpha\partial_\xi^\beta\varphi(y,\xi)} \leq
C_{\alpha\beta}\langle y\rangle^{1-|\alpha|}\jp{\xi}^{1-|\beta|}\quad
(\forall \alpha, \beta),
\\
&\abs{\partial_x^\alpha\partial_y^\beta\partial_\xi^\gamma a(x,y,\xi)}
\leq C_{\alpha\beta\gamma}\langle y\rangle^{-|\beta|}\quad
(\forall \alpha, \beta, \gamma),
\end{align*}
hold for all $x,y,\xi\in\R^n$.
Then $P$ is $L^2(\R^n)$-bounded.
\end{thm}

As for the global $L^p$-boundedness, 
Coriasco and Ruzhansky \cite{CR-CRAS, CR} established the following result
generalising Theorem \ref{Th:RS} to the setting of $L^p$-spaces:

\begin{thm}[\cite{CR}]\label{Th:CR}
Let $\varphi(y,\xi)$ and $a(x,y,\xi)$ be $C^\infty$-functions.
Assume that
 $\varphi(y,\xi)$ is positively homogeneous of order $1$ for large $\xi$
and satisfies \eqref{graphcond}.
Also assume that
\begin{align*}
&\abs{\partial_y^\alpha\partial_\xi^\beta\varphi(y,\xi)} \leq
C_{\alpha\beta}\langle y\rangle^{1-|\alpha|}\jp{\xi}^{1-|\beta|}\quad
(\forall \alpha, \beta),
\\
&\abs{\partial_x^\alpha\partial_y^\beta\partial_\xi^\gamma a(x,y,\xi)}
\leq C_{\alpha\beta\gamma}
\langle x\rangle^{m_1-|\alpha|}
\langle y\rangle^{m_2-|\beta|}
\jp{\xi}^{-(n-1)|1/p-1/2|-|\gamma|}\quad
(\forall \alpha, \beta, \gamma),
\end{align*}
hold for all $x,y,\xi\in\R^n$, and that $a(x,y,\xi)$ vanishes around $\xi=0$.
Then $P$ is $L^p(\R^n)$-bounded, for every $1<p<\infty$, provided that 
$m_1+m_2\leq -n|1/p-1/2|$.
\end{thm}

In Theorem \ref{Th:CR}, the decay of order $-n|1/p-1/2|$ is required
for amplitude functions in space variables. It is also shown in 
\cite{CR} that this order of decay is in general sharp: otherwise it is possible
to construct an example of an operator that is not globally bounded
on $L^p(\R^n)$. 
Thus, in the space $L^p(\R^n)$, in addition to the local loss of regularity of order  
$(n-1)|1/p-1/2|$, there is also the global loss of weight at infinity of order 
$n|1/p-1/2|$, and both of these losses are in general sharp.
It is possible to improve the order of the weight loss a bit 
to the order $(n-1)|1/p-1/2|$
in a special case
of so-called SG-Fourier integral operators, namely, for operators
$$
A u(x)=
\int_{\R^n} e^{i\varphi(x,\xi)}a(x,\xi)\widehat{u}(\xi)\, d\xi,
$$
with amplitudes satisfying
$$
\abs{\partial_x^\alpha\partial_\xi^\gamma a(x,\xi)}
\leq C_{\alpha\gamma}\langle x\rangle^{-(n-1)|1/p-1/2|-|\alpha|}
\jp{\xi}^{-(n-1)|1/p-1/2|-|\gamma|}\quad
(\forall \alpha, \gamma).
$$
If the function $a(x,\xi)$ vanishes near $\xi=0$, such operators are
$L^p(\R^n)$-bounded,
see \cite[Theorem 2.6]{CR} for the precise formulation.

However, despite the mentioned counter-example to the decay order in space
variables, it may be natural to expect some further better properties,
particularly for phase functions whose second derivatives
$\partial_\xi\partial_\xi\varphi$ in $\xi$ are bounded like in the case (I).
For example, for the special case of convolution operators given by
\[
T u(x)=
\int_{\R^n}\int_{\R^n} e^{i((x-y)\cdot\xi-\varphi(\xi))}a(\xi)u(y)\, dy d\xi,
\]
where $\varphi\in C^\infty(\R^n)$
is positively homogeneous of order $1$ for large $\xi$
and $a\in C^\infty(\R^n)$ satisfies
\[
|\partial^\alpha a(\xi)|\leq C_\alpha\jp{\xi}^{\kappa-|\alpha|},
\]
Miyachi \cite{Mi} showed that for $1<p<\infty$, the operator
$T$ is $L^p(\R^n)$-bounded if
$\kappa\leq -(n-1)|1/p-1/2|$
under the assumptions that $\varphi>0$ and that the compact hypersurface
\[
\Sigma=\b{\xi\in\R^n\setminus0\,:\,\varphi(\xi)=1}
\]
has non-zero Gaussian curvature.
Beals \cite{Be} and Sugimoto \cite{Su} discussed the case when
$\Sigma$ might have vanishing Gaussian curvature but is still convex.
But the local $L^p$-boundedness result by Seeger, Sogge and Stein \cite{SSS}
suggests that it could be possible to remove any geometric condition on $\Sigma$.
And indeed, in this paper we establish the following generalised result:
\begin{thm}\label{main}
Let $\varphi(y,\xi)$ and $a(x,y,\xi)$ be $C^\infty$-functions.
Assume that
 $\varphi(y,\xi)$ is positively homogeneous of order $1$ for large $\xi$
and satisfies \eqref{graphcond}.
Also assume that
\begin{align*}
&\abs{\partial_y^\alpha\partial_\xi^\beta(y\cdot\xi-\varphi(y,\xi))}
\le
C_{\alpha\beta}\jp{\xi}^{1-|\beta|}\quad
(\forall\, \alpha, |\beta|\geq1),
\\
&\abs{\partial_x^\alpha\partial_y^\beta\partial_\xi^\gamma a(x,y,\xi)}
\leq C_{\alpha\beta\gamma}
\jp{\xi}^{ -(n-1)|1/p-1/2|-|\gamma|}\quad
(\forall \alpha, \beta, \gamma),
\end{align*}
hold for all $x,y,\xi\in\R^n$.
Then $P$ is $L^p(\R^n)$-bounded, for every $1<p<\infty$.
\end{thm}

Compared to Theorem \ref{Th:CR}, the assumptions on the phase 
$\varphi(y,\xi)$ as in Theorem \ref{main} ensure that no decay of the
amplitude $a(x,y,\xi)$ in the space variables is needed for the operator $P$
to be globally bounded on $L^p(\R^n)$.

\medskip
Theorem \ref{main} together with some related results
will be restated in Section \ref{S2} in a different form
(in particular, Theorem \ref{main} follows from Corollary \ref{Cor1}), emphasising a 
general construction for deducing global results from the local ones.
We now briefly explain the strategy of the global proof.
By the interpolation and the duality,
the problem is reduced to show the
$H^1$-$L^1$-boundedness. 
To show the $H^1$-$L^1$-boundedness, we use the atomic decomposition of $H^1$: 
$$
f=\sum_{j=1}^\infty\lambda_jg_j,\quad\lambda_j\in\C,\quad g_j: \textrm{ atom}.
$$
Here we call a function $g$
on $\R^n$ an atom if there is a ball $B=B_g\subset\R^n$ such that
$\supp g\subset B$, $\|g\|_{L^\infty}\leq|B|^{-1}$ and
$\int g(x)\,dx=0$.
This is the common starting point which is also used to show the known results
\cite{SSS}, \cite{Mi}, \cite{Su}, \cite{CR}.
Modifying these results is not so straightforward but we present 
a new argument which allows to deduce the estimate for large atoms ($|B|\geq1$) from
the global $L^2$-boundedness, and for small atoms ($|B|\leq1$) from the local
$L^p$-boundedness.
More details and proofs are given in Sections \ref{S3} and \ref{S4}.

\medskip
In Section \ref{S5} we give applications of the obtained results to global $L^p$-estimates
for solutions of Cauchy problems for hyperbolic partial differential equations.
In \cite{CR}, the global $L^p$-boundedness of solutions of such equations was established
with a loss of weight at infinity. In Theorem \ref{THM:hyp1} we show that this weight loss
can be eliminated.

\medskip
To complement some references on the local and global boundedness properties of Fourier integral operators,
we refer to 
the authors' paper \cite{RS-weighted:MN} for the weighted $L^2$- and to
Dos Santos Ferreira and Staubach \cite{Staubach-Dos-Santos-Ferreira:local-global-FIOs-MAMS}
for other weighted properties of Fourier integral operators, to
Rodr\'iguez-L\'opez and Staubach \cite{Rodriguez-Lopez-Staubach:rough-FIOs} for estimates for
rough Fourier integral operators, to \cite{Ruzhansky:CWI-book} for
$L^p$-estimates for Fourier integral operators with complex phase functions, as well as to
\cite{Ruzhansky:FIOs-local-global} for an earlier overview of local and global properties of
Fourier integral operators with real and complex phase functions.
$L^{p}$-boundedness of bilinear Fourier integral operators has been also investigated,
see Hong, Lu, Zhang \cite{GLU} and references therein.

\medskip
In this paper we often abuse the notation slightly by writing, for example,
$a(x,y,\xi)\in C^\infty$ instead of $a\in C^\infty$, to emphasise the dependence
on particular sets of variables. We will also often write $\partial_\xi$ for
$\nabla_\xi$.

\section{Main results}
\label{S2}

Let $\mathcal P$ be a Fourier integral operator of the form
\begin{equation}\label{FIOp}
\begin{aligned}
&\mathcal{P} u(x)=
\int_{\R^n}\int_{\R^n} e^{i\phi(x,y,\xi)}a(x,y,\xi)u(y)
\, dy d\xi\quad (x\in\R^n),
\\
&\phi(x,y,\xi)=(x-y)\cdot\xi+\Phi(x,y,\xi),
\end{aligned}
\end{equation}
where $\Phi(x,y,\xi)$ is introduced just for convenience and 
we do not lose any generality with this notation.
We introduce a class for the amplitude $a(x,y,\xi)$.
\begin{defn}\label{symbol}
For $\kappa\in\R$, $S^\kappa$ denotes
the class of smooth functions
$a=a(x,y,\xi)\in C^\infty(\R^n\times\R^n\times\R^n)$
satisfying the estimate
\[
\abs{
\partial^\alpha_x \partial^\beta_y \partial^\gamma_\xi a(x,y,\xi)
}
\leq C_{\alpha\beta\gamma}\jp{\xi}^{\kappa-|\gamma|}
\]
for all $x,y,\xi\in\R^n$ and all multi-indices $\alpha,\beta,\gamma$.
\end{defn}
We remark that the formal adjoint $\mathcal P^*$ of $\mathcal P$ is
of the same form \eqref{FIOp} with the replacement
\begin{equation}\label{adjoint}
\begin{aligned}
&\Phi(x,y,\xi)\longmapsto\Phi^*(x,y,\xi)=-\Phi(y,x,\xi),
\\
&a(x,y,\xi)\longmapsto a^*(x,y,\xi)=\overline{a(y,x,\xi)},
\end{aligned}
\end{equation}
and $a\in S^\kappa$ is equivalent to $a^*\in S^\kappa$.

\medskip
We also introduce a notion of the local boundedness of $\mathcal P$.
By $\chi_K$ we denote the multiplication
by the characteristic function of the set $K\subset\R^n$.
\begin{defn}
We say that the operator $\mathcal P$
is $H^1_{comp}(\R^n)$-$L^1_{loc}(\R^n)$-bounded if the localised operator
$\chi_{K} \mathcal P \chi_{K}$ is $H^1(\R^n)$-$L^1(\R^n)$-bounded
for any compact set $K\subset\R^n$.
Furthermore, if the operator norm of $\chi_{K_h} \mathcal P \chi_{K_h}$
is  bounded in $h\in\R^n$ for the translated set $K_h=\{x+h:x\in K\}$
of any compact set $K\subset\R^n$,
we say that
the operator $\mathcal P$
is uniformly $H^1_{comp}(\R^n)$-$L^1_{loc}(\R^n)$-bounded.
\end{defn}
If we introduce the translation operator $\tau_h:f(x)\mapsto f(x-h)$ and
its inverse (formal adjoint) $\tau_h^*=\tau_{-h}$,
we have the expression $\chi_{K_h}=\tau_h\chi_{K}\tau_h^*$.
Since $L^1$ and $H^1$ norms are translation invariant,
$\mathcal P$
is uniformly $H^1_{comp}$-$L^1_{loc}$-bounded
if and only if $\chi_{K}(\tau_h^* \mathcal P \tau_h)\chi_{K}$ is $H^1$-$L^1$-bounded
for any compact set
$K\subset\R^n$ and the operator norm is bounded in $h\in\R^n$.
We remark that the operator $\tau_h^*\mathcal{P}\tau_h$ is of the form
\eqref{FIOp} with the replacements
\begin{equation}\label{translation}
\begin{aligned}
&\Phi(x,y,\xi)\longmapsto\Phi^h(x,y,\xi)=\Phi(x+h,y+h,\xi),
\\
&a(x,y,\xi)\longmapsto a^h(x,y,\xi)=a(x+h,y+h,\xi).
\end{aligned}
\end{equation}
\par
Now we are ready to state our main principle:
\begin{thm}\label{Th:main}
Assume the following conditions:
\begin{itemize}
\item[(A1)]
$\Phi(x,y,\xi)$ is a real-valued $C^\infty$-function
and $\partial^\gamma_\xi\Phi(x,y,\xi)\in S^0$
for $|\gamma|=1$.
\item[(A2)]
$\mathcal P$ and $\mathcal P^*$ 
are $L^2(\R^n)$-bounded if $a(x,y,\xi)\in S^{0}$.
\item[(A3)]
$\mathcal P$ and $\mathcal P^*$ are uniformly
$H^1_{comp}(\R^n)$-$L^1_{loc}(\R^n)$-bounded
if $a(x,y,\xi)\in S^{-(n-1)/2}$.
\end{itemize}
Then $\mathcal P$ is $L^p(\R^n)$-bounded
if $1<p<\infty$, $\kappa\leq-(n-1)|1/p-1/2|$, and $a(x,y,\xi)\in S^\kappa$.
\end{thm}

We remark that assumptions (A1) and (A2) are essentially the requirements
for phase functions $\Phi(x,y,\xi)$,
and a condition for (A2) is given by Asada and Fujiwara \cite{AF} or by the authors
\cite{RS}, while (A3) is given by Seeger, Sogge and Stein \cite{SSS}.
These conditions will be discussed in Section \ref{S4}, and here we simply state
the final conclusion by
restricting our phase functions to the form
\[
\phi(x,y,\xi)=x\cdot\xi-\varphi(y,\xi)
\quad(\text{in other words
$\Phi(x,y,\xi)=y\cdot\xi-\varphi(y,\xi)$}).
\]
We make precise the notion of homogeneity:

\begin{defn}\label{Def:hom}
We say that
{\it $\varphi(y,\xi)$ is positively homogeneous of order $1$} if
\begin{equation}\label{EQ:hom}
\varphi(y,\lambda\xi)=\lambda\varphi(y,\xi)
\end{equation}
holds for all $y\in\R^n$,
$\xi\neq0$ and $\lambda>0$.
We also say that 
{\it $\varphi(y,\xi)$ is positively homogeneous of order $1$
for large $\xi$}
if there exist a constant $R>0$ such that \eqref{EQ:hom}
holds for all $y\in\R^n$,
$|\xi|\geq R$ and $\lambda\geq1$.
\end{defn}
For the operator of the form
\begin{equation}\label{FIOp1}
P u(x)=
\int_{\R^n}\int_{\R^n} e^{i(x\cdot\xi-\varphi(y,\xi))}a(x,y,\xi)u(y)\, dy d\xi
\quad (x\in\R^n)
\end{equation}
we have:
\begin{cor}\label{Cor1}
Assume the following conditions:
\begin{itemize}
\item[(B1)]
$\varphi(y,\xi)$ is a real-valued $C^\infty$-function
and $\partial^\gamma_\xi(y\cdot\xi-\varphi(y,\xi))\in S^0$
for $|\gamma|=1$.
\item[(B2)]
There exists a constant $C>0$ such that
$\abs{\det \partial_y\partial_\xi\varphi(y,\xi)}\geq C$
for all $y,\xi\in\R^n$.
\item[(B3)]
$\varphi(y,\xi)$ is positively homogeneous
of order $1$ for large $\xi$.
\end{itemize}
Then $P$ is $L^p(\R^n)$-bounded
if $1<p<\infty$, $\kappa\leq-(n-1)|1/p-1/2|$ and $a(x,y,\xi)\in S^\kappa$.
\end{cor}
We can admit positively homogeneous phase functions
which might have singularity
at the origin for a special kind of operators of the form
\begin{equation}\label{FIOp2}
\mathcal T u(x)=\int_{\R^n} e^{i(x\cdot\xi+\psi(\xi))}a(x,\xi)\widehat{u}(\xi)
\, d\xi\quad (x\in\R^n).
\end{equation}
For such operators we have
\begin{cor}\label{Cor2}
Let $1<p<\infty$ and let $\kappa\leq-(n-1)|1/p-1/2|$.
Assume that $a=a(x,\xi)\in S^\kappa$ and
that $\psi=\psi(\xi)$ is a real-valued $C^\infty$-function on $\R^n\setminus0$
which is positively homogeneous
of order $1$.
Then $\mathcal T$ is $L^p(\R^n)$-bounded.
\end{cor}

The proofs of all results in this section will be given in subsequent
sections. The proofs will follow from the global $H^1(\R^n)$-$L^1(\R^n)$-boundedness
of the corresponding operators of order $-(n-1)/2$ by interpolation with
condition (A2) and by duality. Therefore, among other things, in addition to the 
$L^p(\R^n)$-boundedness,
we will obtain that all
the appearing operators of order $-(n-1)/2$ are globally $H^1(\R^n)$-$L^1(\R^n)$-bounded and $L^\infty(\R^n)$-$BMO(\R^n)$-bounded.

\section{Proof of Theorem \ref{Th:main}}\label{S3}
In this section we give the proof of Theorem \ref{Th:main}.
We only have to
show the $L^p$-boundedness of the operator \eqref{FIOp}
assuming $a(x,y,\xi)\in S^\kappa$ with the critical case $\kappa=\kappa(p)$,
where
$$
\kappa(p)=-(n-1)|1/p-1/2|.
$$
On account of the observation \eqref{adjoint} and the invariance of the
assumptions $a=a(x,y,\xi)\in S^\kappa$ and (A1) under such replacement, 
we can restrict our consideration to the case $1<p<2$
by the duality argument.
Furthermore, by assumption (A2) and the complex interpolation argument,
we have only to show the $H^1$-$L^1$-boundedness 
of the operator $\mathcal P$ with $a(x,y,\xi)\in S^{\kappa(1)}$
under the assumptions
\begin{itemize}
\item
(A1),
\item
$\mathcal P$ is $L^2$-bounded,
\item
$\mathcal P$ is uniformly
$H^1_{comp}$-$L^1_{loc}$-bounded.
\end{itemize}
We remark that we still need the global $L^2$-boundedness of $\mathcal P$
to induce the global $H^1$-$L^1$-boundedness form the local one.

First of all, we prepare some useful lemmas. 
We have (at least formally) the kernel representation
\begin{equation}\label{kernelrep}
\mathcal{P} u(x)=
\int_{\R^n}  K(x,y,x-y)u(y)
\, dy,
\end{equation}
where
\begin{equation}\label{kernel}
K(x,y,z)
=\int_{\R^n}e^{i\{z\cdot\xi+\Phi(x,y,\xi)\}}a(x,y,\xi)\,d\xi.
\end{equation}
On account of the singularity set
\[
\Sigma
=\{(x,y,-\partial_\xi\Phi(x,y,\xi))\in\R^n\times\R^n\times\R^n:x,y,\xi\in\R^n\}
\]
of the kernel \eqref{kernel},
we introduce the function
\[
H(x,y,z):=\inf_{\xi\in\R^n}\abs{z+\partial_\xi\Phi(x,y,\xi)}.
\]
Then we have
$\Sigma=\{(x,y,z)\in\R^n\times\R^n\times\R^n:H(x,y,z)=0 \}
=\bigcap_{d>0}(\Delta_d)^c$,
where
\[
\Delta_d:=\{(x,y,z)\in\R^n\times\R^n\times\R^n:H(x,y,z)\geq d\}.
\]
We also introduce
\begin{align*}
&\widetilde H(z):=\inf_{x,y\in\R^n}H(x,y,z)=
\inf_{x,y,\xi\in\R^n}\abs{z+\partial_\xi\Phi(x,y,\xi)},
\\
&\widetilde\Delta_d:=\{z\in\R^n:\widetilde H(z)\geq d\}.
\end{align*}
Clearly we have the monotonicity of $\Delta_d$ and $\widetilde\Delta_d$
in $d>0$, that is,
$\Delta_{d_1}\subset\Delta_{d_2}$,
$\widetilde\Delta_{d_1}\subset\widetilde\Delta_{d_2}$
for $d_1\geq d_2\geq0$.
In the argument below, we frequently use the quantities
\begin{align*}
&M:=\sum_{|\gamma|\leq n+1}\sup_{x,y,\xi\in\R^n}
|\partial^\gamma_\xi a(x,y,\xi)\jp{\xi}^{-(\kappa(1)-|\gamma|)}|,
\\
&N:=\sum_{1\leq|\gamma|\leq n+2}\sup_{x,y,\xi\in\R^n}
|\partial^\gamma_\xi\Phi(x,y,\xi)\jp{\xi}^{-(1-|\gamma|)}|,
\end{align*}
which are finite since $a\in S^{\kappa(1)}$ and
$\partial_\xi^{\gamma}\Phi\in S^0$ for $|\gamma|=1$ by assumption (A1).
\begin{lem}\label{Lem:inside}
Let $r>0$.
Then for $x\in\widetilde\Delta_{2r}$ and $|y|\leq r$
we have
\begin{equation}\label{HandH}
\widetilde H(x)\leq2H(x,y,x-y)
\end{equation}
and $(x,y,x-y)\in\Delta_r$.
\end{lem}
\begin{proof}
For $x\in\widetilde\Delta_{2r}$ and $|y|\leq r$, we have
\begin{align*}
\widetilde H(x)
&\leq H(x,y,x)\leq|x+\partial_\xi\Phi(x,y,\xi)|
\leq|x-y+\partial_\xi\Phi(x,y,\xi)|+|y| 
\\
&\leq|x-y+\partial_\xi\Phi(x,y,\xi)|+\widetilde H(x)/2
\end{align*}
since $\widetilde H(x)\geq 2r$,
hence we have
$\widetilde H(x)\leq2|x-y+\partial_\xi\Phi(x,y,\xi)|$ for all
$\xi\in\R^n$
to conclude \eqref{HandH}.
Since $\widetilde H(x)\geq 2r$ again, the assertion
$(x,y,x-y)\in\Delta_r$ is readily obtained from \eqref{HandH}.
\end{proof}

\begin{lem}\label{Lem:kernel}
The kernel $K(x,y,z)$ is
smooth on $\bigcup_{d>0}\Delta_d$,
and it satisfies
\begin{equation}\label{boundedness}
\n{H(x,y,z)^{n+1} K(x,y,z)}
_{L^\infty(\Delta_d)}\leq C(n,d,M,N),             
\end{equation}
where $C(n,d,M,N)$ is a positive constant depending
only on $n$, $d>0$, $M$ and $N$.
The function $\widetilde H(z)$ satisfies
\begin{equation}\label{integrable}
\n{\widetilde H(z)^{-(n+1)}}_{L^1(\widetilde\Delta_d)}
\leq C(n,d,N),              
\end{equation}
where $C(n,d,N)$ is a positive constant depending only on $n$, $d>0$ and $N$.
\end{lem}
\begin{proof}
The expression \eqref{kernel} is justified by the
integration by parts, and we have
$$
K(x,y,z)=\int_{\R^n}e^{i\{z\cdot\xi+\Phi(x,y,\xi)\}}
\left(L^*\right)^{n+1} a(x,y,\xi)\,d\xi,
$$
where $L^*$ is the transpose of the operator
$$
L=\frac{(z+\partial_\xi\Phi)\cdot\partial_\xi}{i|z+\partial_\xi\Phi|^2}.
$$
Noticing the relation $d\leq H(x,y,z)\leq|z+\partial_\xi\Phi(x,y,\xi)|$
for $(x,y,z)\in\Delta_d$ and $\xi\in\R^n$,
we easily have the property \eqref{boundedness}.
On the other hand,
we have
$$|z|\leq |z+\partial_\xi\Phi(x,y,\xi)|+N$$
for any $x,y\in\R^n$, $\xi\not=0$,
hence $|z|\leq \widetilde H(z)+N$.
Then for $|z|\geq 2N$ we have $|z|\leq \widetilde H(z)+|z|/2$,
hence $|z|\leq 2\widetilde H(z)$,
and the property \eqref{integrable} is obtained from it since
\begin{align*}
\n{\widetilde H(z)^{-(n+1)}}_{L^1(\widetilde\Delta_d)}
&\leq
\n{\widetilde H(z)^{-(n+1)}}_{L^1(\widetilde\Delta_d\cap\b{|z|\leq2N})}+
\n{\widetilde H(z)^{-(n+1)}}_{L^1(\widetilde\Delta_d\cap\b{|z|\geq2N})}
\\
&\leq
d^{-(n+1)}\n{1}_{L^1(|z|\leq2N)}+
2^{n+1}\n{|z|^{-(n+1)}}_{L^1(|z|\geq2N)}.
\end{align*}
The proof is complete.
\end{proof}

\begin{lem}\label{Lem:L1est}
Let $r\geq1$, and let $h\in\R^n$.
Suppose $\supp f\subset \{x\in \R^n:|x|\leq r\}$.
Then we have
\[
\left\|\tau_h^*\mathcal{P}\tau_h f\right\|_{L^1(\widetilde\Delta_{2r})}
\leq C(n,M,N)\n{f}_{L^1},
\]
where $C(n,M,N)$ is a positive constant depending only on $n$, $M$ and $N$.
\end{lem}
\begin{proof}
For $x\in\widetilde\Delta_{2r}$ and $|y|\leq r$, we have
$\widetilde H(x)\leq2H(x,y,x-y)$ and $(x,y,x-y)\in\Delta_r$
by Lemma \ref{Lem:inside}.
Then from the kernel representation \eqref{kernelrep},
we obtain
\begin{align*}
|\mathcal{P} f(x)|
&\leq 2^{n+1}\widetilde H(x)^{-(n+1)}\int_{|y|\leq r}
\abs{H(x,y,x-y)^{n+1}K(x,y,x-y)f(y)}\,dy
\\
&\leq 2^{n+1}\widetilde H(x)^{-(n+1)}
\n{H(x,y,z)^{n+1}K(x,y,z)}_{L^\infty(\Delta_r)}
\n{f}_{L^1}
\end{align*}
for $x\in\widetilde\Delta_{2r}$.
Hence we have by Lemma \ref{Lem:kernel} and the monotonicity
of $\Delta_d$ and $\widetilde\Delta_d$
\begin{align*}
\left\|\mathcal{P} f\right\|_{L^1(\widetilde\Delta_{2r})}
&\leq 2^{n+1}
\left\|\widetilde H(x)^{-(n+1)}\right\|_{L^1(\widetilde\Delta_{2r})}
\n{H(x,y,z)^{n+1}K(x,y,z)}_{L^\infty(\Delta_r)}\n{f}_{L^1}
\\
&\leq 2^{n+1}\left\|\widetilde H(x)^{-(n+1)}
\right\|_{L^1(\widetilde\Delta_{2})}
\n{H(x,y,z)^{n+1}K(x,y,z)}_{L^\infty(\Delta_1)}\n{f}_{L^1}
\\
&\leq 2^{n+1}C(n,2,N)\,C(n,1,M,N)\n{f}_{L^1}.
\end{align*}
On account of the observation \eqref{translation} and the
invariance of the quantities $M$ and $N$
under such replacement, we have the conclusion.
\end{proof}

\begin{lem}\label{Lem:outside}
Let $r\geq1$.
Then we have
$\R^n\setminus\widetilde\Delta_{2r}\subset\b{z:|z|< \p{2+N}\,r}$. 
\end{lem}
\begin{proof}
For $z\in\R^n\setminus\widetilde\Delta_{2r}$,
we have
$\widetilde H(z)=\inf_{x,y,\xi\in\R^n}\abs{z+\partial_\xi\Phi(x,y,\xi)}
<2r$.
Hence, there exist $x_0,y_0,\xi_0\in\R^n$ such that
$$
|z+\partial_\xi\Phi(x_0,y_0,\xi_0)|<2r.
$$
Then we have 
\[
|z|\leq|z+\partial_\xi\Phi(x_0,y_0,\xi_0)|+|\partial_\xi\Phi(x_0,y_0,\xi_0)|
\leq 2r+N\leq \p{2+N}r
\]
since $r\geq1$.
\end{proof}
Now we are ready to prove the $H^1$-$L^1$-boundedness.
We use the characterisation of $H^1$ by the atomic
decomposition proved by Coifman and Weiss \cite{CW}.
That is, any $f\in H^1(\R^n)$ can be
represented as 
$$
f=\sum_{j=1}^\infty\lambda_jg_j,\quad\lambda_j\in\C,\quad g_j:\textrm{ atom},
$$
and the norm $\|f\|_{H^1}$ is equivalent to the norm
$\n{\b{\lambda_j}_{j=1}^\infty}_{\ell^1}=\sum^\infty_{j=1}|\lambda_j|$.
Here we call a function $g$
on $\R^n$ an atom if there is a ball $B=B_g\subset\R^n$ such that
$\supp g\subset B$, $\|g\|_{L^\infty}\leq|B|^{-1}$ ($|B|$ is the
Lebesgue measure of the ball $B$) and
$\int g(x)\,dx=0$.
From this, all we
have to show is the estimate
$$
\left\|\mathcal{P}g\right\|_{L^1(\R^n)}\leq C
$$
with a constant $C>0$ for all atoms $g$.
By an appropriate translation,
it is further reduced to the estimate
$$
\left\|\tau_h^*\mathcal{P}\tau_hf\right\|_{L^1(\R^n)}\leq C,
\quad f\in \A_r,
$$
where $\A_r$
is the set of all functions $f$ on $\R^n$
such that
$$
\supp f\subset B_r=\{x\in\R^n:|x|\leq r\},\quad\|f\|_{L^\infty}\leq |B_r|^{-1},
\quad\int f(x)\,dx=0.
$$
Here an hereafter in this section, $C$ always denotes
a constant which is independent of $h\in\R^n$ and
$0<r<\infty$.

Suppose $f\in\A_{r}$ with $r\geq 1$.
Then we split $\R^n$ into two parts
$\widetilde\Delta_{2r}$ and $\R^n\setminus\widetilde\Delta_{2r}$.
For the part $\widetilde\Delta_{2r}$,
we have by Lemma \ref{Lem:L1est}
\[
\left\|\tau_h^*\mathcal{P} \tau_hf\right\|_{L^1(\widetilde\Delta_{2r})}
\leq C\n{f}_{L^1}\leq C.
\]
For the part $\R^n\setminus\widetilde\Delta_{2r}$,
we have
by Lemma \ref{Lem:outside} and the Cauchy-Schwarz inequality
\begin{align*}
\left\|\tau_h^*\mathcal{P} \tau_hf\right\|_{L^1(\R^n\setminus\widetilde\Delta_{2r})}
&\leq
\|1\|_{L^2(|x|< \p{2+N}r)}
\left\|\tau_h^*\mathcal{P} \tau_hf\right\|_{L^2(\R^n)}
\\
&\leq
Cr^{n/2}\|f\|_{L^2(\R^n)}
\leq C,
\end{align*}
where we have used the assumption that
$\mathcal{P}$ is $L^2$-bounded.

Suppose now $f\in\A_{r}$ with $r\leq 1$.
then we split $\R^n$ into the
parts $\Delta_{2}$ and $\R^n\setminus\Delta_2$.
For the part  $\Delta_{2}$,
we have by Lemma \ref{Lem:L1est} with $r=1$ and the inclusion
$\supp f\subset B_r\subset B_1$
\[
\left\|\tau_h^*\mathcal{P} \tau_hf\right\|_{L^1(\widetilde\Delta_{2})}
\leq C\n{f}_{L^1}\leq C.
\]
For the part $\R^n\setminus\Delta_2$,
we have by Lemma \ref{Lem:outside}
\begin{align*}
\left\|\tau_h^*\mathcal{P} \tau_h f\right\|_{L^1(\R^n\setminus\Delta_{2})}
&\leq
\left\|\tau_h^*\mathcal{P} \tau_h f\right\|_{L^1(|x|<2+N)}
\\
&\leq C\n{f}_{H^1}
\leq C,
\end{align*}
where we have used the fact that $\mathcal{P}$ is uniformly
$H^1_{comp}$-$L^1_{loc}$-bounded.
Now the proof of Theorem \ref{Th:main}
is complete.

\section{Proof of Corollaries \ref{Cor1} and \ref{Cor2}}
\label{S4}

In this section we prove Corollaries \ref{Cor1} and \ref{Cor2}.

\begin{proof}[Proof of Corollary \ref{Cor1}]
Let us induce assumptions (A1)--(A3) of Theorem \ref{Th:main}
from the assumptions (B1)--(B3) of Corollary \ref{Cor1}
for the special case $\phi(x,y,\xi)=x\cdot\xi-\varphi(y,\xi)$,
in other words, $\Phi(x,y,\xi)=y\cdot\xi-\varphi(y,\xi)$.
We remark that (B1) is just an interpretation of assumption (A1).

As for (A2),
a sufficient condition for the $L^2$-boundededness of $\mathcal P$
is known from Asada and Fujiwara \cite{AF}, that is, Theorem \ref{Th:AF}
in Introduction.
On account of the observation
\eqref{adjoint}, $\mathcal P^*$ is also $L^2$-bounded
under the same condition.
In particular,
(A2) is fulfilled if (B1) and (B2) are satisfied.

A sufficient condition for the $H^1_{comp}$-$L^1_{loc}$-boundedness
of $P$ 
is known by the work of Seeger, Sogge and Stein \cite{SSS},
that is, $P$ is $H^1_{comp}$-$L^1_{loc}$-bounded
for $a(x,y,\xi)\in S^{-(n-1)/2}$
if $\varphi(y,\xi)$ is a real-valued $C^\infty$-function
on $\R^n\times(\R^n\setminus0)$ and positively homogeneous of order $1$.
If we carefully trace the argument in \cite{SSS},
we can say that
$\chi_{K} P \chi_{K}$ is $H^1(\R^n)$-$L^1(\R^n)$-bounded
for any compact set $K\subset\R^n$
and its operator norm is bounded by a constant depending only on
$n$, $K$ and quantities
\begin{align*}
&M_\ell=\sum_{|\alpha|+|\beta|+|\gamma|\leq \ell}\sup_{x,y,\xi\in\R^n}
|\partial^\alpha_x\partial^\beta_y
\partial^\gamma_\xi a(x,y,\xi)\jp{\xi}^{(n-1)/2+|\gamma|)}|,
\\
&N_\ell=\sum_{\substack{|\beta|\leq\ell, \\ 1\leq|\gamma|\leq \ell}}
\sup_{\substack{x,y\in\R^n,\\ \xi\neq0}}
|\partial^\beta_y
\partial^\gamma_\xi (y\cdot\xi-\varphi(y,\xi))\abs{\xi}^{-(1-|\gamma|)}|
\end{align*}
with some large $\ell$.
The same is true for $P^*$ if we trace the argument in \cite{St2} instead
but we require (B2) in this case.
Then $P$ and $P^*$ are uniformly $H^1_{comp}$-$L^1_{loc}$-bounded
if $M_\ell$ and $N_\ell$ are finite
since the quantities $M_\ell$ and $N_\ell$ are invariant under the replacement
in \eqref{translation}.

Based on this fact, $P$ and $P^*$ are uniformly
$H^1_{comp}$-$L^1_{loc}$-bounded if $a\in S^{-(n-1)/2}$
under the assumptions (B1)--(B3).
In fact, if we split $a(x,y,\xi)$ into the sum of $a(x,y,\xi)g(\xi)$ and
$a(x,y,\xi)(1-g(\xi))$ with an appropriate smooth cut-off function
$g\in C^\infty_0(\R^n)$ which is equal to $1$ near the origin,
the terms $P_1$ and $P_1^*$ corresponding to $a(x,y,\xi)(1-g(\xi))$ are 
uniformly $H^1_{comp}$-$L^1_{loc}$-bounded by the above observation.
On the other hand, the terms $P_2$ and $P_2^*$ corresponding
to $a(x,y,\xi)g(\xi)$ are 
$L^1$-bounded (hence uniformly $H^1_{comp}$-$L^1_{loc}$-bounded)
because 
\begin{align*}
&P_2u(x)=\int  K(x,y) u(y)\,dy,\quad P_2^*u(x)=\int \overline{K(y,x)} u(y)\,dy,
\\
& K(x,y)
=\int_{\R^n}e^{i(x\cdot\xi-\phi(y,\xi))}
a(x,y,\xi)g(\xi)\,d\xi,
\end{align*}
and the integral kernel $K(x,y)$
is integrable in both $x$ and $y$.
This fact can be verified by the integration by parts
\begin{align*}
K(x,y)
&=(1+|x-y|^2)^{-n} \int_{\R^n}(1-\Delta_\xi)^n e^{i(x-y)\cdot\xi}\cdot
e^{i(y\cdot\xi-\phi(y,\xi))}a(x,y,\xi)g(\xi)\,d\xi
\\
&=(1+|x-y|^2)^{-n}\int_{\R^n} e^{i(x-y)\cdot\xi}\cdot
(1-\Delta_\xi)^n \{e^{i(y\cdot\xi-\phi(y,\xi))}a(x,y,\xi)g(\xi)\}\,d\xi
\end{align*}
followed by the the conclusion
\[
\abs{K(x,y)}\leq C(1+|x-y|^2)^{-n}
\]
because of assumptions (B1), $a\in S^{-(n-1)/2}$, and $g\in C^\infty_0$.

As a conclusion, (A3) is fulfilled if (B1)--(B3) are satisfied, and thus
the proof of Corollary \ref{Cor1} is complete.
\end{proof}


\begin{proof}[Proof of Corollary \ref{Cor2}]
Again we spilt the amplitude $a(x,\xi)$ into the sum of $a(x,\xi)g(\xi)$ and
$a(x,\xi)(1-g(\xi))$ as in the proof of Corollary \ref{Cor1}.
We remark that the operator $\mathcal T$ defined by \eqref{FIOp2}
is the operator $P$ defined by \eqref{FIOp1}
with $\varphi(y,\xi)=y\cdot\xi-\psi(\xi)$ and
$a(x,y,\xi)=a(x,\xi)$ independent of $y$.
For the term $\mathcal T_1$ corresponding to $a(x,\xi)(1-g(\xi))$,
we just apply Corollary \ref{Cor1}.
For the term $\mathcal T_2$ corresponding to $a(x,\xi)g(\xi)$,
we have
\[
\mathcal T_2u(x)
=
\int_{\R^n} e^{i(x\cdot\xi+\psi(\xi))}a(x,\xi)g(\xi)\widehat u(\xi)\, d\xi
=
a(X,D_x)e^{i\psi(D_x)}g(D_x)u(x).
\]
The pseudo-differential operator $a(X,D_x)$ is $L^p$-bounded 
(see Kumano-go and Nagase \cite{KN})
and the Fourier multiplier $e^{i\psi(D_x)}g(D_x)$ is also $L^p$-bounded
by the Marcinkiewicz theorem (see Stein \cite{St}) 
since
$\abs{\partial^\alpha \p{e^{i\psi(\xi)}g(\xi)}}
\leq C_\alpha|\xi|^{-|\alpha|}$ for any multi-index $\alpha$.
The proof of Corollary \ref{Cor2} is complete.
\end{proof}

\section{Applications to hyperbolic equations}
\label{S5}

In this section we briefly outline an application of the obtained results to the global
$L^p$-estimates for solutions to the Cauchy problems for strictly hyperbolic partial
differential equations. In particular,
in \cite{CR}, the global $L^p$-boundedness of solutions of such equations was established
with a loss of weight at infinity. In Theorem \ref{THM:hyp1} we show that this loss of weight
can be eliminated.

\medskip
For simplicity, we consider equation of the first order   
\begin{equation}
     \left\{ \begin{array}{ll}
         (D_t+a(t,x,D_x) u(t,x)=0, & t\not=0, \, x\in\R^n, \\
         u(0,x)=f(x),
      \end{array} \right.
    \label{eq:Cauchy}
  \end{equation}
 where, as usual, $D_t=-i\partial_t$ and $D_{x}=-i\partial_{x}$.
 We assume that
 the symbol $a(t,x,\xi)$ is a classical symbol with real-valued principal part such that
 \begin{equation}\label{EQ:apps-symbol}
  |\partial_t^k \partial_x^\beta \partial_\xi^\alpha a(t,x,\xi)|\leq C_{k\alpha\beta}
  \jp{\xi}^{1-|\alpha|}
 \end{equation}  
 holds for all $x,\xi\in\R^n$, all $t\in [0,T]$ 
 for some $T>0$,
 and all $k,\alpha,\beta$, with constants $C_{k\alpha\beta}$ independent
 of $t,x,\xi$.

 We consider strictly hyperbolic equations which
 means that the principal symbol of $a(t,x,\xi)$ is real-valued. 

 We note that following Kumano-go \cite{Kumano-go:BOOK-pseudos} we can extend the
 conclusions below also to higher order equations, especially if we impose
 appropriate conditions on lower order terms to achieve the perfect diagonalisation 
 of the corresponding hyperbolic system to keep the phase function in the 
 required form, similarly to the SG-case as in Coriasco \cite{Coriasco:SG-FIOs-II}.

First we note that it was shown by Seeger, Sogge and Stein  \cite{SSS} that if
we have the Sobolev space data $f\in L^{p}_{\alpha+(n-1)|1/p-1/2|}(\R^n)$, 
 for some $\alpha\in\R$, then for each fixed time $t$ the solution satisfies
 $u(t,\cdot)\in L^p_{\alpha}(\R^n)$ locally, $1<p<\infty$. 
 Moreover, this order is
 sharp for every $t$ in the complement of a discrete set
 in $\R$ provided that $a$
 is elliptic in $\xi$. 
 
Let us now outline that Theorem \ref{main}
implies that this result holds globally on $\R^n$.
Under the assumption \eqref{EQ:apps-symbol},
it follows from Kumano-go
 \cite[Ch. 10, \S 4]{Kumano-go:BOOK-pseudos} 
 that  for sufficiently small times, 
 the solution $u(t,x)$ to the Cauchy problem
 \eqref{eq:Cauchy} can be constructed as a Fourier integral operator
 in the form \eqref{EQ:P}. 
 Moreover, it follows from \cite[Ch. 10, Theorem 4.1]{Kumano-go:BOOK-pseudos} 
 that the phase and the amplitude of the propagator
 satisfy assumptions of Theorem \ref{main}. 
 Consequently, we obtain:
 
 \begin{thm}\label{THM:hyp1}
 Let the symbol $a(t,x,\xi)$ satisfy conditions
 \eqref{EQ:apps-symbol}. Let $1<p<\infty$.
 If $f$ is such that $f\in L^p_{(n-1)|1/p-1/2|}(\R^n)$,
 then for each $t\in [0,T]$, the solution $u(t,x)$ of
 the Cauchy problem \eqref{eq:Cauchy} satisfies
 $u(t,\cdot)\in L^p(\R^n)$. Moreover,
 for every $\alpha\in\R$ and $m\in\R$,
 there is $C_T>0$ such that we have the estimate
 \begin{equation}\label{EQ:est-hyp-2}
 \|u(t,\cdot)\|_{L^p_{\alpha}(\R^n)}\leq C_T
 \|f\|_{L^p_{\alpha+(n-1)|1/p-1/2|}(\R^n)},
 \end{equation} 
 for all $t\in [0,T]$ and all $f$ such that the right hand side norm is
 finite.
 \end{thm} 
 
 In particular, Theorem \ref{THM:hyp1} eliminates the weight loss in the global estimates
 estimates for solutions as it was obtained in \cite[Theorems 5.1 and 5.2]{CR}.
 The transition between Sobolev spaces for obtaining estimate \eqref{EQ:est-hyp-2}
 for all $\alpha$ can be done by using the global calculus of Fourier integral operators
 developed by the authors in \cite {RS-weighted:MN}.



\end{document}